\documentclass[11pt]{article}

\usepackage{amsmath,amsthm,amssymb,amsfonts}
\usepackage{mathtools}
\usepackage{geometry}
\usepackage{enumitem}
\usepackage{hyperref}

\newtheorem{theorem}{Theorem}[section]
\newtheorem{proposition}[theorem]{Proposition}
\newtheorem{lemma}[theorem]{Lemma}
\newtheorem{corollary}[theorem]{Corollary}
\newtheorem{definition}[theorem]{Definition}

\theoremstyle{remark}
\newtheorem{remark}{Remark}

\newcommand{\R}{\mathbb R}

\newcommand{\sub}{\subset}
\newcommand{\eps}{\varepsilon}

\title{
Schanuel Integration and Euler Characteristic
of Semi-algebraic Sets
}

\author{Yuxuan Xiao\thanks{Department of Mathematics,
HKUST.
Email: yxiaoce@connect.ust.hk}
}

\date{}

\begin{document}

\maketitle

\begin{abstract}

We extend the Schanuel integration framework, originally introduced for finite unions of convex sets, to arbitrary semi-algebraic sets. 
We prove that the resulting Schanuel integral of indicator functions is independent of the choice of ordered linear bases and therefore defines a well-defined Euler characteristic in the semi-algebraic category.

We further show that this Schanuel--Euler characteristic coincides with the classical Euler characteristic defined via Borel--Moore homology and cylindrical algebraic decomposition. 
The recursive fiberwise structure of Schanuel integration provides an elementary and geometric interpretation of Euler characteristic and yields simplified proofs of several classical properties, including invariance under semi-algebraic isomorphisms.

\end{abstract}

\medskip


\medskip


\section{Introduction}

Semi-algebraic sets play a central role in real algebraic geometry, tame topology and related areas of mathematics. They also arise naturally in optimization, computational geometry, data analysis and mathematical modeling.

Among the topological invariants associated with semi-algebraic sets, the Euler characteristic occupies a distinguished position because of its additivity, geometric invariance and compatibility with decomposition methods.

One of the fundamental results in tame geometry, due to van den Dries, states that two semi-algebraic sets are semi-algebraically isomorphic if and only if they have the same dimension and Euler characteristic~\cite{van1998tame}. 
This demonstrates that the Euler characteristic is not merely a topological invariant, but also a fundamental invariant in the semi-algebraic category.

There are two classical approaches to defining the Euler characteristic of semi-algebraic sets.

The first approach is homological. For a locally compact semi-algebraic set $A$, one defines
\[
\chi(A)=\sum_i(-1)^i\dim_{\mathbb Z/2}H_i^{\mathrm{BM}}(A;\mathbb Z/2),
\]
where $H_i^{\mathrm{BM}}(A;\mathbb Z/2)$ denotes the Borel--Moore homology of $A$ with coefficients in $\mathbb Z/2$~\cite{coste2005real}. This definition extends to arbitrary semi-algebraic sets by means of compactifications and open embeddings.

The second approach is decomposition-theoretic and is based on cylindrical algebraic decomposition (CAD). Every semi-algebraic set admits a decomposition into finitely many cells, each semi-algebraically homeomorphic to an open cube $(0,1)^k$, and one defines
\[
\chi(A)=\sum_i(-1)^{\dim C_i}.
\]
Van den Dries proved that this definition is independent of the chosen cell decomposition~\cite{van1998tame}.

Although these two approaches are well established, they rely respectively on homological machinery and global cell decompositions. It is therefore natural to ask whether the Euler characteristic admits a more elementary and recursive construction based directly on the geometry of fibers under projections.

The purpose of this paper is to answer this question by introducing an alternative approach based on recursive fiberwise integration.

Roughly speaking, let $A\subset\R^n$ be a semi-algebraic set and let $\mathbf 1_A$ denote its indicator function. Given an ordered basis $(v_1,\dots,v_n)$ of $\R^n$, we integrate $\mathbf 1_A$ recursively along the successive projections determined by the basis. For one-dimensional fibers, the integration assigns the value $-1$ to an open interval and $1$ to a singleton. More generally, if the fibers admit a finite semi-algebraic partition on which this construction is well defined, the integral is extended by additivity. Iterating this procedure along the ordered basis produces the Schanuel integral of $\mathbf 1_A$.

This viewpoint originates from valuation theory and the work of
Hadwiger~\cite{hadwiger1955eulers},
Groemer~\cite{groemer1972eulersche},
and Schanuel~\cite{schanuel2006length,schanuel2006negative}.
In particular, Chen introduced a Schanuel-type integration framework for finite unions of convex sets, where the Schanuel integral can be regarded as an iterated integration process of indicator functions~\cite{chen1993euler}.

Compared with the homological and decomposition-theoretic approaches, Schanuel integration has a particularly elementary and geometric character. 
Its recursive nature makes computations inductive on dimension and provides a direct interpretation of the Euler characteristic through the topology of fibers.

The principal difficulty is that the recursive integration process appears to depend on the chosen ordered basis. Our main result shows that this dependence is only apparent.

\begin{theorem}
Let $A\subset\R^n$ be a semi-algebraic set. Then the Schanuel integral of $\mathbf 1_A$ is independent of the choice of ordered bases of $\R^n$.
\end{theorem}

Consequently, Schanuel integration induces a well-defined Euler characteristic on semi-algebraic sets, which we call the \emph{Schanuel--Euler characteristic}.

We further prove that the Schanuel--Euler characteristic coincides with the classical Euler characteristic defined via Borel--Moore homology and cylindrical algebraic decomposition.
Moreover, once the well-definedness of the Schanuel integral has been established, several fundamental properties of the Euler characteristic admit short and transparent proofs within this recursive framework, illustrating another advantage of the Schanuel integration approach.
Hence Schanuel integration provides an alternative construction of the Euler characteristic together with a recursive and geometrically transparent viewpoint.

The paper is organized as follows.
In Section~2, we review Schanuel integration and its basic properties.
In Section~3, we recall cylindrical algebraic decompositions adapted to finite
families of polynomials.
In Section~4, we prove that the Schanuel--Euler characteristic of
semi-algebraic sets is well defined.
Finally, in Section~5, we establish its fundamental invariance properties.

\section{Schanuel Integration}

The following definitions are based on Chen's work~\cite{chen1993euler}.

\begin{definition}
A function $f:\R\to\R$ is said to be Schanuel integrable if:

\begin{enumerate}[label=(\roman*)]
\item $f$ is discontinuous at only finitely many points;
\item $f$ admits left and right limits at every point of $\R$;
\item the limits $f(+\infty)$ and $f(-\infty)$ exist.
\end{enumerate}

\noindent The Schanuel integral of $f$ is defined by
\[
\int_\eps f
=
\sum_{x\in\R}
\left(
f(x)
-\frac12 f(x-)
-\frac12 f(x+)
\right)
-\frac12 f(+\infty)
-\frac12 f(-\infty).
\]
\end{definition}

For example,
\[
\int_\eps \mathbf 1_{\{x\}} = 1,
\qquad
\int_\eps \mathbf 1_{(a,b)} = -1.
\]

\noindent We now define iterated Schanuel integration on $\R^n$.

\begin{definition}
Let $f:\R^n\to\R$.
We say that $f$ is Schanuel integrable with respect to the ordered basis $(e_n,\dots,e_1)$ if the recursive sequence of functions
\[
I_n^f=f,
\]
and
\[
I_{k-1}^f(a_1,\dots,a_{k-1})
=
\int_\eps
I_k^f(a_1,\dots,a_{k-1},x_k)\,dx_k
\]
is well-defined for all $1\le k\le n$.
\end{definition}

\noindent Informally,
\[
I_k^f(a_1,\dots,a_k)
=
\int_\eps\cdots\int_\eps
f(a_1,\dots,a_k,x_{k+1},\dots,x_n)
\,dx_n\cdots dx_{k+1}.
\]

\noindent One of the basic property of the Schanuel integral is linearity, which indicates the additivity of the Schanuel-Euler characteristic.

\begin{proposition}
If $f,g:\R^n\to\R$ are Schanuel integrable, then so are $f+g$ and $\lambda f$ for $\lambda\in\R$, and
\[
\int_\eps (f+g)
=
\int_\eps f
+
\int_\eps g,
\]
\[
\int_\eps \lambda f
=
\lambda \int_\eps f.
\]
\end{proposition}

Let $(v_n,\dots,v_1)$ be an ordered basis of $\R^n$ and let
\[
T_v:\R^n\to\R^n
\]
be the linear isomorphism satisfying
\[
T_v(e_i)=v_i.
\]

\begin{definition}
A function $f:\R^n\to\R$ is Schanuel integrable with respect to $(v_n,\dots,v_1)$ if
\[
f\circ T_v
\]
is Schanuel integrable with respect to the standard basis.

In this case we define
\[
\int_{\eps_{(v_n,\dots,v_1)}} f
:=
\int_\eps (f\circ T_v).
\]
\end{definition}

\begin{definition}
A function $f:\R^n\to\R$ is called Schanuel integrable if its Schanuel integral is independent of the choice of ordered linear basis.

The Schanuel--Euler characteristic of $f$ is defined by
\[
\chi(f)
:=
\int_\eps f.
\]
\end{definition}

In particular, if $A\sub\R^n$ is a subset and $f=\mathbf 1_A$, we define
\[
\chi(A):=\chi(\mathbf 1_A).
\]

It is straightforward to prove by induction that the Schanuel integral is well defined for relatively open convex sets and convex bodies. In particular,
\[
\int_\varepsilon \mathbf 1_U=(-1)^{\dim U},
\qquad
\int_\varepsilon \mathbf 1_K=1.
\]

Chen established the well-definedness of the Schanuel--Euler characteristic for finite unions of convex sets ~\cite{chen1993euler}. 
The present paper extends this theory from convex geometry to the semi-algebraic category. Our main theorem proves that the Schanuel--Euler characteristic is well defined for arbitrary semi-algebraic sets.

\section{Semi-algebraic Sets and CAD}

In this section we recall the form of cylindrical algebraic decomposition
(CAD) that will be used later, and explain why it implies the
Schanuel-integrability of indicator functions of semi-algebraic sets.  The
discussion follows the standard treatment in semi-algebraic geometry; see, for
example, Coste's notes~\cite{coste2002introduction}.

Let \(C \subseteq \mathbb{R}^{k}\) be a semi-algebraic set, and let
\[
  \xi_1 < \cdots < \xi_r : C \longrightarrow \mathbb{R}
\]
be continuous semi-algebraic functions.  These functions decompose the cylinder
\(C \times \mathbb{R}\) into two types of semi-algebraic pieces.  The
\emph{graph cells} are the sets
\[
  \Gamma(\xi_i)
  :=
  \{(x,t)\in C\times \mathbb{R} : t=\xi_i(x)\},
  \qquad 1\leq i\leq r,
\]
while the \emph{band cells} are the regions between consecutive graphs, together
with the two unbounded regions:
\[
  \{(x,t)\in C\times \mathbb{R}: t<\xi_1(x)\},
\]
\[
  \{(x,t)\in C\times \mathbb{R}: \xi_i(x)<t<\xi_{i+1}(x)\},
  \qquad 1\leq i<r,
\]
and
\[
  \{(x,t)\in C\times \mathbb{R}: \xi_r(x)<t\}.
\]
Each graph cell is semi-algebraically homeomorphic to \(C\), and each band cell
is semi-algebraically homeomorphic to \(C\times (0,1)\).

Let \(L\subseteq \mathbb{R}[X_1,\ldots,X_n]\) be a finite collection of
polynomials.  A subset \(C\subseteq \mathbb{R}^n\) is called
\emph{\(L\)-invariant} if, for every \(f\in L\), the sign of \(f\) is constant
on \(C\).  Equivalently, for every \(f\in L\), the map
\[
  x\longmapsto \operatorname{sgn}(f(x))
\]
is constant on \(C\), where the sign takes values in \(\{-1,0,1\}\).

Given a finite set
\[
  L=\{f_1,\ldots,f_r\}\subseteq \mathbb{R}[X_1,\ldots,X_n],
\]
CAD associates to \(L\) another finite set
\[
  \operatorname{PROJ}(L)\subseteq \mathbb{R}[X_1,\ldots,X_{n-1}],
\]
We will not need its explicit construction.  The
properties relevant for us are the following:
\begin{enumerate}

\item The sign conditions determined by \(\operatorname{PROJ}(L)\) give a
finite semi-algebraic decomposition of \(\mathbb{R}^{n-1}\).  For every
connected component \(C\) of such a sign-condition set, there exist continuous
semi-algebraic functions
\[
  \xi_{C,1}<\cdots<\xi_{C,r_C}:C\longrightarrow \mathbb{R}
\]
such that, for each \(x\in C\), the set
\[
  \{\xi_{C,1}(x),\ldots,\xi_{C,r_C}(x)\}
\]
is precisely the set of real roots, in the variable \(X_n\), of the polynomials
in \(L\) after specializing
\[
  (X_1,\ldots,X_{n-1})=x.
\]
\item the cylinder \(C\times \mathbb{R}\) decomposes into the graph and band
cells determined by the functions \(\xi_{C,i}\), and each such cell is
\(L\)-invariant.  
\item the projection construction is monotone: if
\(L\subseteq L_1\), then
\[
  \operatorname{PROJ}(L)\subseteq \operatorname{PROJ}(L_1).
\]
\end{enumerate}

Iterating this projection-and-lifting procedure yields the following form of
the cylindrical algebraic decomposition theorem.

\begin{theorem}[Cylindrical algebraic decomposition]
\label{thm:cad}
Let
\[
  L\subseteq \mathbb{R}[X_1,\ldots,X_n]
\]
be a finite set of polynomials.  Then, for each \(k=1,\ldots,n\), there is a
finite partition \(\mathcal{C}_k\) of \(\mathbb{R}^k\) into connected
semi-algebraic sets, called cells, such that the following hold.

\begin{enumerate}
  \item Every cell \(C\in \mathcal{C}_1\) is either a point or an open interval.

  \item For every \(1\leq k<n\) and every \(C\in \mathcal{C}_k\), there are
  finitely many continuous semi-algebraic functions
  \[
    \xi_{C,1}<\cdots<\xi_{C,r_C}:C\longrightarrow \mathbb{R}
  \]
  such that the graph and band cells determined by these functions are exactly
  the cells of \(\mathcal{C}_{k+1}\) lying over \(C\).

  \item Each cell is semi-algebraically homeomorphic to an open cube.

  \item If \(L^{(0)}:=L\) and
  \[
    L^{(j+1)}:=\operatorname{PROJ}(L^{(j)}),
  \]
  then every cell \(C\in \mathcal{C}_k\) is \(L^{(n-k)}\)-invariant.
\end{enumerate}
\end{theorem}

\section{Schanuel--Euler Characteristic of Semi-algebraic Sets}
We now apply CAD to semi-algebraic sets.  Let
\(A\subseteq \mathbb{R}^n\) be semi-algebraic.  Choose a finite family of
polynomials \(L\subseteq \mathbb{R}[X_1,\ldots,X_n]\) such that \(A\) is
described by a Boolean combination of sign conditions on elements of \(L\).
By Theorem~\ref{thm:cad}, there is a CAD adapted to \(L\).  Since every cell is
\(L\)-invariant, the truth value of the defining Boolean formula for \(A\) is
constant on each cell.  Hence \(A\) is a finite union of cells in this CAD.

This observation is the bridge between CAD and the Schanuel--Euler
characteristic.

\begin{proposition}
\label{prop:semialgebraic-indicator-schanuel-integrable}
Let \(A\subseteq \mathbb{R}^n\) be a semi-algebraic set.  Then, with respect to
any ordered linear basis \((v_1,\ldots,v_n)\) of \(\mathbb{R}^n\), the
indicator function \(\mathbf{1}_A\) is Schanuel-integrable.
\end{proposition}

\begin{proof}
Fix an ordered linear basis \((v_1,\ldots,v_n)\) of \(\mathbb{R}^n\), and let
\[
  T:\mathbb{R}^n\longrightarrow \mathbb{R}^n,
  \qquad
  (x_1,\ldots,x_n)\longmapsto \sum_{i=1}^n x_i v_i .
\]
To prove Schanuel-integrability with respect to this basis, it is enough to
prove the corresponding statement for \(T^{-1}(A)\) in the standard coordinate
basis.  Since linear isomorphisms preserve semi-algebraic sets, \(T^{-1}(A)\)
is semi-algebraic.  Thus we may assume that the chosen basis is the standard
one.

Choose a finite family of polynomials \(L\) such that \(A\) is defined by a
Boolean combination of sign conditions on elements of \(L\), and take a CAD
adapted to \(L\).  Since each cell is \(L\)-invariant, the truth value of the
defining Boolean formula for \(A\) is constant on every cell.  Hence \(A\) is a
finite disjoint union of cells,
\[
  A = C_1\sqcup \cdots \sqcup C_m .
\]
Therefore
\[
  \mathbf{1}_A = \sum_{i=1}^m \mathbf{1}_{C_i}.
\]

It remains to note that the indicator function of every CAD cell is
Schanuel-integrable.  We prove this by induction on the ambient dimension.
For \(n=1\), every cell is either a point or an open interval, whose Schanuel
integrals are respectively \(1\) and \(-1\).

Assume the claim is known in dimension \(n-1\), and let
\(C\subseteq \mathbb{R}^n\) be a cell lying over a cell
\(D=\pi_n(C)\subseteq \mathbb{R}^{n-1}\).  If \(C\) is a graph cell over \(D\),
then integration in the last coordinate gives
\[
  \int_{\varepsilon} \mathbf{1}_C\,dx_n
  =
  \mathbf{1}_D .
\]
If \(C\) is a band cell over \(D\), then
\[
  \int_{\varepsilon} \mathbf{1}_C\,dx_n
  =
  -\mathbf{1}_D .
\]
By the induction hypothesis, \(\mathbf{1}_D\) is Schanuel-integrable, and hence
so is \(\mathbf{1}_C\).  Moreover the same induction gives
\[
  \int_{\varepsilon} \mathbf{1}_C
  =
  (-1)^{\dim C}.
\]
Thus every \(\mathbf{1}_{C_i}\) is Schanuel-integrable.  By finite additivity,
\(\mathbf{1}_A\) is Schanuel-integrable.
\end{proof}

We next prove that the Schanuel--Euler characteristic obtained in this way is
independent of the chosen ordered linear basis.  We first record a small
regularity fact for root functions.

\begin{lemma}
\label{lem:local-extension-root-function}
Let \(C\subseteq \mathbb{R}^{n-1}\) be connected and locally path-connected,
and let
\[
  \xi:C\longrightarrow \mathbb{R}
\]
be a continuous function such that \(\xi(x)\) is a simple root of a polynomial
\(g\in \mathbb{R}[X_1,\ldots,X_n]\) for every \(x\in C\).  Then for every
\(x_0\in C\), there is an open neighbourhood \(W\subseteq\mathbb{R}^{n-1}\) of
\(x_0\) and a smooth function
\[
  \xi':W\longrightarrow \mathbb{R}
\]
such that
\[
  \xi'|_{W\cap C}=\xi|_{W\cap C}.
\]
Moreover, for \(1\leq k\leq n-1\),
\[
  \frac{\partial \xi'}{\partial X_k}
  =
  -
  \left(
    \frac{\partial g}{\partial X_n}
  \right)^{-1}
  \frac{\partial g}{\partial X_k},
\]
where the right-hand side is evaluated at \((x,\xi'(x))\).
\end{lemma}

\begin{proof}
For every \(x\in C\), we have
\[
  \frac{\partial g}{\partial X_n}(x,\xi(x))\neq 0.
\]
By the implicit function theorem, for every \(x_0\in C\), there exists an open
neighbourhood \(W\subseteq \mathbb{R}^{n-1}\) of \(x_0\) and a smooth function
\(\xi':W\to\mathbb{R}\) such that
\[
  \xi'(x_0)=\xi(x_0),
  \qquad
  g(x,\xi'(x))=0
\]
for all \(x\in W\).  Since \(C\) is locally path-connected, we may shrink \(W\)
so that \(W\cap C\) is connected.

Set
\[
  C_W:=\{x\in W\cap C:\xi(x)=\xi'(x)\}.
\]
This set is nonempty and closed in \(W\cap C\).  We claim that it is also open.
Let \(x\in C_W\), and put \(y=\xi(x)=\xi'(x)\).  Since \(y\) is a simple root
of \(g(x,T)\), the implicit function theorem gives a neighbourhood
\(U\subseteq \mathbb{R}^n\) of \((x,y)\) such that
\[
  g^{-1}(0)\cap U
  =
  \Gamma(\xi')\cap U .
\]
Let \(G_\xi:C\to\mathbb{R}^n\) be the graph map
\(G_\xi(z)=(z,\xi(z))\).  Then
\[
  V:=G_\xi^{-1}(U)
\]
is an open neighbourhood of \(x\) in \(C\), and on \(V\cap W\) we have
\(\xi=\xi'\).  Hence \(C_W\) is open in \(W\cap C\).  Since \(W\cap C\) is
connected and \(C_W\) is nonempty, open, and closed, we get
\(C_W=W\cap C\).  The derivative formula is the usual one obtained by
differentiating \(g(x,\xi'(x))=0\).
\end{proof}

The following remark is the key ingredient in proving the commutativity of two successive Schanuel integrations.
\begin{remark}
\label{rem:monotone-root-functions}
Let \(L\subseteq \mathbb{R}[x,y]\) be finite and closed under
\(\partial/\partial y\), and let \(C\subseteq\mathbb{R}\) be a connected cell in
a CAD adapted to \(L\).  Suppose \(\xi:C\to\mathbb{R}\) is a root function
arising from \(L\).  If \(\xi\) has multiplicity \(m\) as a root of
\(f\in L\), then \(\xi\) is a simple root of
\[
  g_\xi
  :=
  \frac{\partial^{m-1}f}{\partial y^{m-1}}.
\]
Thus
\[
  \frac{\partial g_\xi}{\partial y}(x,\xi(x))
  =
  \frac{\partial^m f}{\partial y^m}(x,\xi(x))
  \neq 0
\]
for all \(x\in C\).  If, in addition, the CAD is adapted to all
\(\partial g_\xi/\partial x\), then Lemma~\ref{lem:local-extension-root-function}
implies that the sign of \(\xi'\) is constant on each open interval cell.
Consequently every such root function is either constant or strictly monotone
on each open interval cell.
\end{remark}

We will use the following reduction due to Chen~\cite{chen1993euler}, which shows that basis independence can be tested by invariance under permutations of a fixed ordered basis.

\begin{proposition}[Permutation reduction]
\label{prop:permutation-reduction}

Let $f:\mathbb R^n\to\mathbb R$ be a function. Suppose that for every ordered basis
$v=(v_1,\ldots,v_n)$ of $\mathbb R^n$ and every permutation
$\sigma\in S_n$,
\[
\int_{v}f=\int_{v_\sigma}f,
\]
where
\[
v_\sigma=(v_{\sigma(1)},\ldots,v_{\sigma(n)}).
\]
Then $\int_{v}f$ is independent of the choice of the ordered basis $v$.

\end{proposition}

\begin{theorem}
\label{thm:basis-independence-schanuel-euler}
Let \(A\subseteq\mathbb{R}^n\) be semi-algebraic.  Then
\[
  \int_{\varepsilon_{(v_n,\ldots,v_1)}} \mathbf{1}_A
\]
is independent of the ordered basis \((v_n,\ldots,v_1)\).
\end{theorem}

\begin{proof}
By Proposition~\ref{prop:semialgebraic-indicator-schanuel-integrable},
the iterated Schanuel integral is defined for every ordered basis.

Since every ordered basis is obtained from the standard basis by a linear automorphism of $\mathbb R^n$, 
and linear automorphisms preserve semi-algebraic sets, it suffices to consider the standard basis. 
By Proposition~\ref{prop:permutation-reduction}, it is therefore enough to prove that the iterated Schanuel integral is invariant under permutations of the coordinate order.

Every permutation is a product of adjacent transpositions.  Hence it suffices
to prove that
\[
  \int_{\varepsilon} \mathbf{1}_A\,
  dx_n\cdots dx_{i+1}\,dx_i\cdots dx_1
  =
  \int_{\varepsilon} \mathbf{1}_A\,
  dx_n\cdots dx_i\,dx_{i+1}\cdots dx_1
\]
for every adjacent pair \(i,i+1\).  The integrations in the coordinates other
than \(x_i,x_{i+1}\) are the same on both sides and produce a finite
\(\mathbb{Z}\)-linear combination of indicator functions of semi-algebraic
sets.  By finite additivity, the problem reduces to the two-dimensional case:
for every semi-algebraic \(A\subseteq\mathbb{R}^2\), one must show
\[
  \int_{\varepsilon}\mathbf{1}_A\,dx\,dy
  =
  \int_{\varepsilon}\mathbf{1}_A\,dy\,dx .
\]

We prove this two-dimensional statement.  Choose a finite set
\(L\subseteq\mathbb{R}[x,y]\) such that \(A\) is defined by a Boolean
combination of sign conditions on elements of \(L\).  Enlarging \(L\) if
necessary, we may assume that \(L\) is closed under \(\partial/\partial y\).
Now further enlarge \(L\) to a finite set \(L_1\):
$$ L_1= L\cup \left\{\frac{\partial f}{\partial x}: 
\text{if }\frac{\partial f}{\partial x} \text{ is not constant for $f\in L$}\right\} $$
Take a CAD of \(\mathbb{R}^2\)
adapted to \(L_1\).  By monotonicity of the $\text{PROJ}$ construction, this CAD 
refines the CAD adapted to \(L\), and hence \(A\) is a finite disjoint union of
two-dimensional CAD cells for \(L_1\).

It is therefore enough to prove the equality for a single cell
\(K\subseteq\mathbb{R}^2\).  Let \(C=\pi_x(K)\subseteq\mathbb{R}\), where
\(\pi_x(x,y)=x\).  The cell \(K\) is either a graph cell or a band cell over
\(C\).  Let \(d\) be the fiber dimension of \(K\) over \(C\); thus \(d=0\) for a
graph cell and \(d=1\) for a band cell.  By the definition of the Schanuel
integral in the \(y\)-direction,
\[
  \int_{\varepsilon}\mathbf{1}_K\,dy\,dx
  =
  (-1)^d
  \int_{\varepsilon}\mathbf{1}_C\,dx
  =
  (-1)^{d+\dim C}.
\]

We now compute the integral in the opposite order.  There are three cases.

First, suppose \(C\) is a point.  Then \(\dim C=0\).  The projection of \(K\)
to the \(y\)-axis is either a point or an open interval according as \(K\) is a
graph cell or a band cell.  Hence
\[
  \int_{\varepsilon}\mathbf{1}_K\,dx\,dy
  =
  (-1)^d
  =
  \int_{\varepsilon}\mathbf{1}_K\,dy\,dx .
\]

Second, suppose \(C\) is an open interval and \(K\) is a band cell over \(C\).
Write
\[
  K=\{(x,y)\in C\times\mathbb{R}:\xi_1(x)<y<\xi_2(x)\},
\]
allowing \(\xi_1=-\infty\) or \(\xi_2=+\infty\) for unbounded bands.  By
Remark~\ref{rem:monotone-root-functions}, after passing to the refined CAD the
finite root functions are constant or strictly monotone on \(C\).  Therefore
every nonempty horizontal fiber of \(K\) is an open interval.  Since the
projection of \(K\) to the \(y\)-axis is also an interval, we get
\[
  \int_{\varepsilon}\mathbf{1}_K\,dx\,dy
  =
  -\int_{\varepsilon}\mathbf{1}_{\pi_y(K)}\,dy
  =
  (-1)(-1)
  =
  1.
\]
On the other hand,
\[
  \int_{\varepsilon}\mathbf{1}_K\,dy\,dx
  =
  (-1)^{1+1}
  =
  1.
\]

Third, suppose \(C\) is an open interval and \(K\) is the graph of a root
function \(\xi:C\to\mathbb{R}\).  Again, \(\xi\) is either constant or strictly
monotone on \(C\).  If \(\xi\) is strictly monotone, then every nonempty
horizontal fiber is a point and \(\pi_y(K)\) is an open interval.  Hence
\[
  \int_{\varepsilon}\mathbf{1}_K\,dx\,dy
  =
  \int_{\varepsilon}\mathbf{1}_{\pi_y(K)}\,dy
  =
  -1.
\]
If \(\xi\) is constant, then \(\pi_y(K)\) is a point and the horizontal fiber is
the open interval \(C\).  Hence
\[
  \int_{\varepsilon}\mathbf{1}_K\,dx\,dy
  =
  -\int_{\varepsilon}\mathbf{1}_{\pi_y(K)}\,dy
  =
  -1.
\]
In both subcases,
\[
  \int_{\varepsilon}\mathbf{1}_K\,dx\,dy
  =
  -1
  =
  (-1)^{0+1}
  =
  \int_{\varepsilon}\mathbf{1}_K\,dy\,dx .
\]

Thus the two orders of integration agree on every cell \(K\).  Since \(A\) is a
finite disjoint union of such cells and the Schanuel integral is finitely
additive,
\[
  \int_{\varepsilon}\mathbf{1}_A\,dx\,dy
  =
  \sum_K
  \int_{\varepsilon}\mathbf{1}_K\,dx\,dy
  =
  \sum_K
  \int_{\varepsilon}\mathbf{1}_K\,dy\,dx
  =
  \int_{\varepsilon}\mathbf{1}_A\,dy\,dx .
\]
This proves invariance under adjacent transpositions, hence under all
permutations of the coordinate order, and therefore the value of
\[
  \int_{\varepsilon_{(v_n,\ldots,v_1)}} \mathbf{1}_A
\]
is independent of the ordered basis.
\end{proof}

\section{Other Conclusions}
The preceding results allow us to define the Euler characteristic of a
semi-algebraic set by Schanuel integration.  Namely, if
\(A\subseteq \mathbb{R}^n\) is semi-algebraic, we set
\[
  \chi(A)
  :=
  \int_{\varepsilon} \mathbf{1}_A .
\]
By Theorem~\ref{thm:basis-independence-schanuel-euler}, this number is
independent of the ordered linear basis used to compute the iterated Schanuel
integral.  We now record several basic consequences.  These properties show
that this definition agrees with the usual Euler characteristic of semi-algebraic sets.

One advantage of the Schanuel integration framework is that several fundamental properties of the Euler characteristic admit particularly transparent proofs. 
In particular, the invariance under semi-algebraic isomorphisms follows directly from the well-definedness of the iterated Schanuel integral.
\begin{proposition}[Invariance under semi-algebraic isomorphisms]
\label{prop:semialgebraic-isomorphism-preserves-euler}
Let
\[
  \psi:A\xrightarrow{\cong}B
\]
be a semi-algebraic isomorphism, where
\(A\subseteq \mathbb{R}^n\) and \(B\subseteq \mathbb{R}^m\) are
semi-algebraic sets.  Then
\[
  \chi(A)=\chi(B).
\]
\end{proposition}

\begin{proof}
Consider the graph
\[
  \Gamma_\psi
  :=
  \{(x,\psi(x))\in \mathbb{R}^n\times \mathbb{R}^m:x\in A\}.
\]
Write
\[
  \psi(x)=(\psi_1(x),\ldots,\psi_m(x)).
\]
We compute the Euler characteristic of \(\Gamma_\psi\) using the ordered
coordinate basis
\[
  (e_{n+m},\ldots,e_{n+1},e_n,\ldots,e_1).
\]
For each \(x\in A\), the fiber of \(\Gamma_\psi\) over \(x\) in the last
\(m\) coordinates is a single point.  Therefore successive Schanuel integration
in the coordinates \(e_{n+m},\ldots,e_{n+1}\) gives
\[
  \int_{\varepsilon}
  \mathbf{1}_{\Gamma_\psi}\,
  de_{n+m}\cdots de_{n+1}
  =
  \mathbf{1}_A .
\]
Hence
\[
  \chi(\Gamma_\psi)
  =
  \int_{\varepsilon}
  \mathbf{1}_{\Gamma_\psi}\,
  de_{n+m}\cdots de_1
  =
  \int_{\varepsilon}
  \mathbf{1}_A\,
  de_n\cdots de_1
  =
  \chi(A).
\]

Applying the same argument to \(\psi^{-1}:B\to A\), we get
\[
  \chi(\Gamma_{\psi^{-1}})=\chi(B).
\]
The graphs \(\Gamma_\psi\) and \(\Gamma_{\psi^{-1}}\) differ only by a
permutation of the two coordinate blocks.  Thus, by the independence of the
ordered basis,
\[
  \chi(\Gamma_\psi)=\chi(\Gamma_{\psi^{-1}}).
\]
Consequently
\[
  \chi(A)=\chi(\Gamma_\psi)=\chi(\Gamma_{\psi^{-1}})=\chi(B).
\]
\end{proof}

\begin{corollary}
\label{cor:cell-euler-characteristic}
Let \(C\) be a \(k\)-cell in a CAD, equivalently a semi-algebraic cell
semi-algebraically isomorphic to an open cube \((0,1)^k\).  Then
\[
  \chi(C)=(-1)^k.
\]
\end{corollary}

\begin{proof}
By Proposition~\ref{prop:semialgebraic-isomorphism-preserves-euler}, it is
enough to compute the Euler characteristic of \((0,1)^k\).  Successive Schanuel integration gives one factor of \(-1\) for each open
interval coordinate.  Hence
\[
  \chi((0,1)^k)=(-1)^k.
\]
\end{proof}

It follows that if a semi-algebraic set \(A\) is decomposed into finitely many
CAD cells \(C\in\mathcal A\), then
\[
\chi(A)
=
\sum_{C\in\mathcal A}(-1)^{\dim C}.
\]

On the other hand, the Borel--Moore Euler characteristic is additive with
respect to finite decompositions into locally closed semi-algebraic sets and
takes the value \((-1)^k\) on every open \(k\)-cell~\cite{coste2005real}. Hence it is given by the
same cell-counting formula. Consequently, the Schanuel--Euler characteristic
coincides with the classical Euler characteristic defined via
Borel--Moore homology.

\begin{corollary}
\label{cor:compact-simplicial-euler-characteristic}
Let \(A\) be a compact semi-algebraic set, and let \(|K|\) be a finite
simplicial complex semi-algebraically homeomorphic to \(A\).  Then
\[
  \chi(A)
  =
  \sum_{\sigma\in K} (-1)^{\dim \sigma}.
\]
\end{corollary}

\begin{proof}
By triangulation of compact semi-algebraic sets, \(A\) is
semi-algebraically homeomorphic to \(|K|\).  By
Proposition~\ref{prop:semialgebraic-isomorphism-preserves-euler}, it is enough
to compute \(\chi(|K|)\).  Decompose \(|K|\) into its relatively open simplices.
Each relatively open simplex \(\sigma^\circ\) is a semi-algebraic
\(\dim\sigma\)-cell, and hence
\[
  \chi(\sigma^\circ)=(-1)^{\dim\sigma}
\]
by Corollary~\ref{cor:cell-euler-characteristic}.  By finite additivity,
\[
  \chi(|K|)
  =
  \sum_{\sigma\in K}\chi(\sigma^\circ)
  =
  \sum_{\sigma\in K}(-1)^{\dim\sigma}.
\]
\end{proof}

\begin{proposition}[Coproduct formula]
\label{prop:coproduct-formula}
Let \(A\) and \(B\) be disjoint semi-algebraic sets.  Then
\[
  \chi(A\sqcup B)=\chi(A)+\chi(B).
\]
\end{proposition}

\begin{proof}
Since \(A\) and \(B\) are disjoint,
\[
  \mathbf{1}_{A\sqcup B}
  =
  \mathbf{1}_A+\mathbf{1}_B .
\]
The result follows from the additivity of Schanuel integration.
\end{proof}

\begin{corollary}[Inclusion--exclusion]
\label{cor:inclusion-exclusion}
Let \(A\) and \(B\) be semi-algebraic sets.  Then
\[
  \chi(A\cup B)
  =
  \chi(A)+\chi(B)-\chi(A\cap B).
\]
\end{corollary}

\begin{proof}
We have the disjoint decompositions
\[
  A\cup B
  =
  (A\setminus B)\sqcup (B\setminus A)\sqcup (A\cap B),
\]
\[
  A
  =
  (A\setminus B)\sqcup (A\cap B),
  \qquad
  B
  =
  (B\setminus A)\sqcup (A\cap B).
\]
Applying Proposition~\ref{prop:coproduct-formula} to these decompositions gives
the desired identity.
\end{proof}

\begin{proposition}[Product formula]
\label{prop:product-formula}
Let \(A\subseteq \mathbb{R}^n\) and
\(B\subseteq \mathbb{R}^m\) be semi-algebraic sets.  Then
\[
  \chi(A\times B)=\chi(A)\chi(B).
\]
Equivalently,
\[
  \int_{\varepsilon}\mathbf{1}_{A\times B}\,
  dx_{n+m}\cdots dx_1
  =
  \left(
    \int_{\varepsilon}\mathbf{1}_A\,dx_n\cdots dx_1
  \right)
  \left(
    \int_{\varepsilon}\mathbf{1}_B\,dx_m\cdots dx_1
  \right).
\]
\end{proposition}

\begin{proof}
Choose a CAD of \(B\), and write \(B\) as a finite disjoint union of cells
\[
  B=C_1\sqcup\cdots\sqcup C_r.
\]
Then
\[
  A\times B
  =
  \bigsqcup_{i=1}^r A\times C_i.
\]
By finite additivity, it is enough to compute the Schanuel integral of
\(\mathbf{1}_{A\times C_i}\).

Fix a cell \(C=C_i\).  In the product coordinates on
\(\mathbb{R}^{n+m}=\mathbb{R}^n\times\mathbb{R}^m\), first integrate along the
\(B\)-coordinates.  Since \(C\) is a CAD cell, successive integration along
these coordinates gives
\[
  \int_{\varepsilon}\mathbf{1}_{A\times C}\,
  dx_{n+m}\cdots dx_{n+1}
  =
  \chi(C)\mathbf{1}_A
  =
  (-1)^{\dim C}\mathbf{1}_A .
\]
Therefore
\[
  \int_{\varepsilon}\mathbf{1}_{A\times C}\,
  dx_{n+m}\cdots dx_1
  =
  (-1)^{\dim C}
  \int_{\varepsilon}\mathbf{1}_A\,dx_n\cdots dx_1 .
\]
Summing over all cells of \(B\), we obtain
\[
\begin{aligned}
  \int_{\varepsilon}\mathbf{1}_{A\times B}\,
  dx_{n+m}\cdots dx_1
  &=
  \sum_{i=1}^r
  (-1)^{\dim C_i}
  \int_{\varepsilon}\mathbf{1}_A\,dx_n\cdots dx_1        \\
  &=
  \left(\sum_{i=1}^r (-1)^{\dim C_i}\right)
  \int_{\varepsilon}\mathbf{1}_A\,dx_n\cdots dx_1        \\
  &=
  \chi(B)\chi(A).
\end{aligned}
\]
This proves the product formula.
\end{proof}

\section*{Acknowledgements}

The author would like to thank Prof.\ Chen Beifang for introducing the Schanuel integration framework and for many valuable discussions.
The author also thanks Cao Ying and Wang Hongyang for helpful conversations related to this work.

\bibliographystyle{plain}
\bibliography{references}

\end{document}